\documentclass[12pt,amstex]{amsart}
\usepackage[latin9]{inputenc}
\usepackage{amstext}
\usepackage{amsthm}
\usepackage{amssymb}

\makeatletter
\numberwithin{equation}{section}
\numberwithin{figure}{section}
\theoremstyle{plain}

  \theoremstyle{plain}

\usepackage{amssymb}
\usepackage{amsmath}

\theoremstyle{definition} %%% for statements in roman typeface

\newtheorem*{remark}{Remark}

\theoremstyle{plain}  %%% for statements in italic typeface

\newtheorem{theorem}{Theorem}[section]
\newtheorem{proposition}[theorem]{Proposition}
\newtheorem{corollary}[theorem]{Corollary}
\newtheorem{lemma}[theorem]{Lemma}

\DeclareMathOperator{\id}{id}

\makeatother

\providecommand{\lemmaname}{Lemma}
\providecommand{\theoremname}{Theorem}

\begin{document}

\title{Weighted cogrowth formula for free groups}

\author{Johannes Jaerisch}
\address{Department of Mathematics, Interdisciplinary Faculty of Science and Engineering,\endgraf 
Shimane University,\endgraf
Matsue, Shimane 690-8504, Japan}
\email{jaerisch@riko.shimane-u.ac.jp}

\author{Katsuhiko Matsuzaki}
\address{Department of Mathematics, School of Education, Waseda University,\endgraf
Shinjuku, Tokyo 169-8050, Japan}
\email{matsuzak@waseda.jp}

\thanks{This work was supported by JSPS KAKENHI 16K13767.}

\begin{abstract} We investigate the relationship between geometric, analytic and probabilistic indices for  quotients of  the Cayley graph of the free group  ${\rm Cay}(F_n)$ %endowed  with variable edge lengths, 
 by an arbitrary subgroup $G$ of $F_n$. %The geometric index is given by the Poincar\'e exponent of $G$, the analytic index is the bottom of the spectrum of weighted discrete Laplacian. 
Our main result, which generalizes Grigorchuk's cogrowth formula   to  variable edge lengths, provides a formula relating the bottom of the spectrum of weighted  Laplacian on $G \backslash {\rm Cay}(F_n)$ to  the Poincar\'e exponent of $G$.  Our main tool is the  Patterson-Sullivan theory for Cayley graphs with variable edge lengths. 
%In particular, we obtain a formula for the bottom of the spectrum of the Laplacian on ${\rm Cay}(F_n)$ in terms of the Poincar\'e exponent of $F_n$.
\end{abstract}

\maketitle

\section{Introduction and statement of results}

Let  $F_n=\langle a_1,\ldots,a_n\rangle$ denote the free group of rank $n\ge 2$ and let ${\rm Cay}(F_n)$ denote its Cayley graph. For an arbitrary  subgroup $G \subset F_n$,   the action of $G$ on   ${\rm Cay}(F_n)$ defines the  quotient graph $G \backslash {\rm Cay}(F_n)$.
In this paper, we compare  fundamental indices of geometric, analytic and probabilistic  nature associated with $G$ acting on ${\rm Cay}(F_n)$. The geometric index is the Poincar\'e exponent $\delta_G$ given by   the exponential growth rate
of  $G$-orbits 
$$
\delta_G=\limsup_{R \to \infty} \frac{\log \#\{g \in G \mid d({\rm id}, g) \leq R\}}{R},
$$
where  $d$ denotes the metric on  $F_n$ giving each  edge of ${\rm Cay}(F_n)$ the length one. 
%We have $\delta_{F_n}=\log (2n-1)$ because 
%$$
%\#\{g \in G \mid d({\rm id}, g) \leq R\}=2n \cdot (2n-1)^{R-1} \quad (R \in \N).
%$$
The analytic  index is the bottom of the spectrum of the Laplacian $\Delta=I-A$ on
$L^2(G \backslash {\rm Cay}(F_n))$  denoted by $\lambda_0^G$. 
Here,  $I$ denotes the identity matrix and $A$ the  transition matrix  of the simple random walk on ${\rm Cay}(F_n)$, which is   for  each  function $f$ on the vertex set of $G \backslash {\rm Cay}(F_n)$ given by
$$
(Af)(x)=\frac{1}{2n}\{f(xa_1)+f(xa_1^{-1})+\cdots +f(xa_n)+f(xa_n^{-1})\} \quad (x \in G \backslash {\rm Cay}(F_n)).
$$

 The  two indices, geometric and analytic,  are related by the following well-known formula. Note that the edge lengths of ${\rm Cay}(F_n)$ and the weights of $A$ are constant.
\begin{theorem}[Grigorchuk's cogrowth formula \cite{MR599539, MR1436550}]\label{cogrowth}
\[
\lambda_0^G=
\begin{cases} \frac{1}{2n}\,(2n-1-e^{\delta_G})(1-e^{-\delta_G}) & (\delta_G > \frac{1}{2}\log (2n-1))\\
1-\frac{\sqrt{2n-1}}{2} & (\delta_G \leq \frac{1}{2}\log (2n-1))
\end{cases}.
\]
\end{theorem}
That $\lambda_0:=\lambda_0^{\{\rm id\}}=1-{\sqrt{2n-1}}/{2}$ follows from earlier work of Kesten (\cite{MR0109367}) who proved that  the spectral radius of $A$ is equal to the decay rate of the return probabilities of the simple random walk on ${\rm Cay}(F_n)$.  Also note that $\delta_{F_n}=\log (2n-1)$, so that $\lambda_0$ is related to  $\delta_{F_n}/2$.  
 Related results for discrete groups acting on  hyperbolic space were obtained by
Elstrodt, Patterson and Sullivan in \cite{MR882827}. The case of pinched negative curvature was recently considered in \cite{MR3350111}.
%It is easy to see that $\delta_{F_n}=\log (2n-1)$.

In this paper, we consider the case of
variable edge lengths 
of ${\rm Cay}(F_n)$. 
For any ${\bf r}=(r_1,\ldots,r_n)$ with $r_1+ \cdots +r_n=1/2$ and $r_i>0$ for all $i$,
we define the length of the edge corresponding to the generator $a_i^{\pm}$
to be $-\log r_i$ for all $i$.
The Cayley graph
${\rm Cay}(F_n)$ equipped with this distance $d_{\bf r}$ is denoted by $X_{\bf r}$.

Any subgroup of $G \subset F_n$ acts on $X_{\bf r}$ isometrically, properly discontinuously, and freely.
The Poincar\'e exponent $\delta_G({\bf r})$ of $G$ acting on $X_{\bf r}$ is defined in the same manner.
In our normalization of the edge length, the even length case with $r_i=1/(2n)$ for all $i$ gives 
$\delta_{F_n}({\bf r})=\log (2n-1)/\log (2n)$. 
%A natural problem is to compute $\delta({\bf r}):=\delta_{F_n}({\bf r})$ for any ${\bf r}$.
Unlike the case of equal edge lengths, even in the special case $G=F_n$, the value of  $\delta({\bf r}):=\delta_{F_n}({\bf r})$   
is unclear  in the variable edge length setting, since it is not easy to 
count $\#\{g \in F_n \mid d_{\bf r}({\rm id}, g) \leq R\}$ directly. We will consider the problem to compute $\delta({\bf r})$ in Theorem \ref{main1} below. 

We also consider variable weights for the discrete  Laplacian. 
For every ${\bf p}=(p_1,\ldots,p_n)$ with $p_1+ \cdots +p_n=1/2$ and $p_i>0$ for all $i$,
the stochastic transition matrix $A_{\bf p}=(p(x,y))_{x,y}$ for vertices $x,y \in F_n$ of ${\rm Cay}(F_n)$ is given by
$p(x,y)=p_i$ if $y=xa_i^{\pm}$. This defines an operator which is, for each function $f$ on the vertex set of ${\rm Cay}(F_n)$, given by 
\[
(A_{\bf p}f)(x):=\sum_{i=1}^n p_i(f(xa_i)+f(xa_i^{-1})).
\]
The weighted Laplacian is then defined by $\Delta_{\bf p}:=I-A_{\bf p}$. 

For a subgroup $G \subset F_n$, the Laplacian $\Delta_{\bf p}$ acts on $L^2 (G \backslash {\rm Cay}(F_n))$
as a bounded symmetric operator. 
The bottom of the spectrum  of $\Delta_{\bf p}$ is denoted by $\lambda_0^G({\bf p})$.
Since $A_{\bf p}$ is also a bounded symmetric operator with non-negative entries,
the spectral radius $\rho^G({\bf p})$
of $A_{\bf p}$ coincides with its operator norm, and this is also given by
\[
\rho^G({\bf p})=\sup \langle A_{\bf p} f,f \rangle,
\]
where $\langle \cdot, \cdot \rangle$ is the inner product of $L^2 (G \backslash {\rm Cay}(F_n))$ and
the supremum is taken over all $f \in L^2 (G \backslash {\rm Cay}(F_n))$ with $\langle f, f \rangle=1$.
Then, we have that
\[
\lambda_0^G({\bf p})=1-\rho^G({\bf p})=\inf \langle \Delta_{\bf p} f,f \rangle.
\]
It is easy to see that $\rho^{F_n}({\bf p})=1$ and 
$\lambda_0^{F_n}({\bf p})=0$ for every ${\bf p}$.
%It is known that $\lambda_0^G({\bf p})$ is the maximum of eigenvalues $\lambda$ satisfying $\Delta_{\bf p} f=\lambda f$
%for a positive eigenfunction $f$ on $G \backslash {\rm Cay}(F_n)$.

Concerning $\rho({\bf p}):=\rho^{\{\rm id\}}({\bf p})$ and 
$\lambda_0({\bf p}):=\lambda_0^{\{\rm id\}}({\bf p})=1-\rho({\bf p})$,
 the following formula is well-known:  
 \begin{equation}\label{Woess}
 \rho({\bf p})=\min_{t>0} \frac{1}{t} \left\{\sum_{i=1}^n \sqrt{1+4p_i^2t^2}-(n-1)\right\}.
 \end{equation}
 The formula (\ref{Woess}) is a special case of  \cite{MR0442698}. The case $n=2$ was considered in \cite{MR0461671}. Further references can be found in \cite{MR1743100}.  See Section 9 and the Notes at the end of  Chapter II of this book for details. We will also obtain this formula in the course of our arguments.  Moreover, we will express $\rho({\bf p})$ in a different way by using the Poincar\'e exponent of $F_n$  (see Theorem \ref{main2} below).
%\begin{theorem}[\cite{MR1743100}]\label{Woess}
%\[
%\rho({\bf p})=\min_{t>0} \frac{1}{t} \left\{\sum_{i=1}^n \sqrt{1+4p_i^2t^2}-(n-1)\right\}.
%\]
%\end{theorem}

We investigate the problems mentioned above for the variable parameters.
Our method is to find the proper correspondence between the edge length parameter $\bf r$
and the weight $\bf p$ for the Laplacian. To obtain eigenfunctions of the Laplacian $\Delta_{\bf p}$, we use an integral representation by the Patterson measure instead of the integral of the Martin kernel. An idea of choosing weights of the Laplacian from Patterson measures can be found in \cite{MR1425084}.

Since ${\rm Cay}(F_n)$ is a tree, $X_{\bf r}=({\rm Cay}(F_n),d_{\bf r})$ is a Gromov $0$-hyperbolic space. 
Given a boundary point $\xi \in \partial X_{\bf r}$, 
we define $j_{\bf r}(x,\xi)=\exp(-b_{\xi}(x))$ for every vertex $x \in X_{\bf r}$, where $b_{\xi}$ is the Busemann
function with respect to the geodesic ray $\beta_{\xi}:[0,\infty) \to X_{\bf r}$
from the base point $o=\beta_{\xi}(0)$ to $\xi=\beta_{\xi}(\infty)$ given by
$$
b_{\xi}(x)=\lim_{t \to \infty}(t-d_{\bf r}(x,\beta_{\xi}(t)).
$$ 
For the Laplacian $\Delta_{\bf p}$ of weight $\bf p$, the eigenrelation
\begin{equation}\label{fundamental}
\Delta_{\bf p} j_{\bf r}(x,\xi)^s=\lambda j_{\bf r}(x,\xi)^s \quad (\forall \xi \in \partial X_{\bf r})
\end{equation}
with $\lambda \in \mathbb R$ and $s \in (0,1)$ gives the correspondence between $\bf r$ and $\bf p$.
This can be explicitly given in the following way.

We set the spaces of parameters
\begin{align*}
\mathcal R&:=\{{\bf r}=(r_1,\ldots,r_n) \mid r_1+\cdots +r_n=1/2,\ r_i>0\ (\forall i)\};\\
\mathcal P&:=\{{\bf p}=(p_1,\ldots,p_n) \mid p_1+\cdots +p_n=1/2,\ p_i>0\ (\forall i)\}.
\end{align*}
We also define a diffeomorphism $H:\mathcal R \times (0,\infty) \to (0,1)^n$ by $H({\bf r},s)={\bf u}:=(u_1,\ldots,u_n)$, 
$u_i=r_i^s$. Under this transformation, relation (\ref{fundamental}) turns out to be
\begin{equation}
\lambda=1-2\sum_{k=1}^nu_{k}p_{k}-(u_{i}^{-1}-u_{i})p_{i} \quad (i=1,\ldots,n).
\end{equation}
Solving these equations for unknown variables ${\bf p}=(p_1,\ldots,p_n)$ and $\lambda$ by linear algebra, we have functions
${\bf p}({\bf u})$ and $\lambda({\bf u})$ if the determinant is not zero. On the other hand, given $\bf p \in \mathcal P$ and 
$\lambda \geq 0$,
we can obtain a solution 
${\bf u} \in (0,1)^n$ by using the Green function of the random walk on ${\rm Cay}(F_n)$ if $\lambda \leq \lambda_0({\bf p})$.

The following theorem, which will be proved in Section 2,  allows us to compute the Poincar\'e exponent.

\begin{theorem}\label{main1}
For every ${\bf r}=(r_1,\ldots,r_n) \in {\mathcal R}$, the Poincar\'e exponent $\delta({\bf r})$ of $F_n$ satisfies
the equation $\lambda \circ H({\bf r},s)=0$ for $s=\delta({\bf r})$. More precisely, $\delta({\bf r})$ is
the unique solution $s \in (0,1)$ of the equation
\[
\sum_{i}r_{i}^s+3\sum_{(i_{1},i_{2})}(r_{i_{1}}r_{i_{2}})^s+5\sum_{(i_{1},i_{2},i_{3})}(r_{i_{1}}r_{i_{2}}r_{i_{3}})^s+\cdots+(2n-1)(r_{1}\cdots r_{n})^s=1,
\]
where the subscript $(i_1,\ldots,i_m)$ represents taking all indices satisfying $i_1< \cdots <i_m$.
\end{theorem}

Related to the  formula in Theorem \ref{Woess}, we will prove the following result in Section 3. 
A novelty of our result is that we determine the minimum in (\ref{Woess}) 
by using the Poincar\'e exponent of $F_n$ acting on the weighted Cayley graph.

\begin{theorem}\label{main2}
To each ${\bf p} \in {\mathcal P}$, there corresponds a unique ${\bf r} \in \mathcal R$
such that the bottom of  the spectrum of  $\lambda_0(\bf p)$
of the Laplacian $\Delta_{\bf p}$ on ${\rm Cay}(F_n)$ is given by $\lambda \circ H({\bf r},\delta({\bf r})/2)$.
\end{theorem}

From this theorem, we can expect that the appropriate weight ${\bf p}_{*}({\bf r},s)$ for the Laplacian
is given by 
\[
{\bf p}_{*}({\bf r},s):=\begin{cases}
{\bf p}\circ H({\bf r},s) & (s>\delta({\bf r})/2)\\
{\bf p}\circ H({\bf r},\delta({\bf r})/2) & (s\leq\delta({\bf r})/2).
\end{cases}
\]
In Section 4, we generalize Grigorchuk's cogrowth formula in the following form. 
This is our  main result of this paper.
\begin{theorem}\label{main3}
For any subgroup $G\subset F_{n}$ and for any ${\bf r} \in \mathcal R$,
the bottom of the spectrum  $\lambda_{0}^G({\bf p}_{*}({\bf r},\delta_{G}({\bf r})))$ of 
the Laplacian $\Delta_{{\bf p}_{*}({\bf r},\delta_{G}({\bf r}))}$
on the quotient graph $G \backslash {\rm Cay}(F_n)$ is given by 
\begin{equation}\label{eps}
\lambda_{0}^G({\bf p}_{*}({\bf r},\delta_{G}({\bf r})))=\begin{cases}
\lambda \circ H({\bf r},\delta_{G}({\bf r})) & (\delta_{G}({\bf r})>\delta({\bf r})/2)\\
\lambda \circ H({\bf r},\delta({\bf r})/2) & (\delta_{G}({\bf r})\leq\delta({\bf r})/2).
\end{cases}
\end{equation}
\end{theorem}
We recall from \cite{MR0109367, MR0112053} that if $N$ is a normal subgroup of $F_n$ then $\lambda_0^N(\bf p)$ is equal to zero for any $\bf p$ if and only if $F_n/N$ is amenable. Combining this characterization with Theorem  \ref{main3} applied to
${\bf p}={\bf p} \circ H({\bf r}, \delta({\bf r}))$,  we obtain the following amenability criterion. The corollary below was proved in \cite{MR3240930}  using different methods. In the case of  equal edge lengths, the corollary is  Grigorchuk's amenability criterion. An alternative proof is given in  \cite{MR2338235}.   For Kleinian groups a related result is due to Brooks (\cite{MR783536}). 

\begin{corollary}[Weighted cogrowth criterion for amenability] Let $N$ be a normal subgroup of $F_n$. Then the weighted cogrowth $\delta_N({\bf r})/\delta_{F_n}({\bf r})$ is equal to one if and only if $F_n/N$ is amenable.
\end{corollary}

\section{A computation of the Poincar\'e exponent}

The Cayley graph ${\rm Cay}(F_{n})$ of the free group $F_{n}=\langle a_{1},a_{2},\ldots,a_{n}\rangle$
is the regular tree of valency $2n$. For any positive real numbers $r_{1},r_{2},\ldots,r_{n}>0$
with the normalization $r_{1}+r_{2}+\cdots+r_{n}=1/2$, we assign length $-\log r_{i}$
to the edges of labels $a_{i}$ and $a_{i}^{-1}$ in ${\rm Cay}(F_{n})$
for $i=1,2,\ldots,n$. 
We regard this proper metric space as a Gromov
hyperbolic space and represent it by $X_{{\bf r}}$ with the distance $d_{{\bf r}}$
for every
\[
{\bf r} \in \mathcal R:=\{{\bf r}=(r_1,\ldots,r_n) \mid r_1+\cdots +r_n=1/2,\ r_i>0\ (\forall i)\}. 
\]
The free group $F_{n}$
acts on $X_{{\bf r}}$ isometrically, properly discontinuously and
cocompactly. We choose the vertex $\id$ of ${\rm Cay}(F_{n})$ as
the base point $o$ of $X_{{\bf r}}$.

For a vertex $x\in X_{{\bf r}}$ and $\xi\in\partial X_{{\bf r}}$, set $j(x,\xi)=\exp(-b_{\xi}(x))$,
where $b_{\xi}$ is the Busemann function with respect to the geodesic
ray from $o$ to $\xi$. For $s\geq0$, a conformal measure of dimension
$s$ is a family of positive finite Borel measures $\{\mu_{x}\}_{x\in X_{{\bf r}}}$
on $\partial X_{{\bf r}}$ such that 
\[
\frac{d\mu_{x}}{d\mu_{y}}(\xi)=\left(\frac{j(x,\xi)}{j(y,\xi)}\right)^{s}
\]
for any vertices $x,y\in X_{{\bf r}}$. For a subgroup $G\subset F_{n}$, the
conformal measure $\{\mu_{x}\}_{x\in X_{{\bf r}}}$ is $G$-invariant
if $\mu_{g(x)}(g(E))=\mu_{x}(E)$ for every vertex $x\in X_{{\bf r}}$ and
for every Borel subset $E\subset\partial X_{{\bf r}}$. For any $G$-invariant
conformal measure $\mu=\{\mu_{x}\}_{x\in X_{{\bf r}}}$ of dimension
$s$, the total mass function 
\[
\varphi_{\mu}(x)=\int_{\partial X_{{\bf r}}}d\mu_{x}=\int_{\partial X_{{\bf r}}}j(x,\xi)^{s}d\mu_{o}(\xi)
\]
is $G$-invariant.

For any subgroup $G\subset F_{n}$, the exponent of convergence is
defined by 
\[
\delta_{G}({\bf r})=\limsup_{R\to\infty}\frac{\log\#\{g\in G\mid d_{{\bf r}}(o,g(o))\leq R\}}{R}.
\]
A $G$-invariant conformal measure of dimension
$\delta_{G}({\bf r})$ is called a Patterson measure for $G$. The results on the Patterson measure for a discrete group acting on a Gromov hyperbolic space
can be summarized as follows in our particular situation.

\begin{theorem}[Coornaert \cite{MR1214072}]\label{patterson}
For every subgroup $G \subset F_n$, there exists a $G$-invariant conformal measure $\mu$
of dimension $\delta_G({\bf r})$.
If $G$ is finitely generated, then it is unique up to constant multiples.
\end{theorem}

\begin{remark}
It is well known that $G$ is convex cocompact if and only if $G$ is finitely generated (\cite{MR1468105, MR1170365, MR1804703}).
\end{remark}

For every ${\bf p}=(p_{1},p_{2},\ldots,p_{n})$ with $p_{1}+p_{2}+\cdots+p_{n}=1/2$ and $p_i>0$ for all $i$,
we define a transition matrix $A_{\bf p}=(p(x,y))_{x,y}$ on the vertices of ${\rm Cay}(F_{n})$
by $p(x,y)=p_{i}$ if $y=xa_{i}$ or $y=xa_{i}^{-1}$ for $i=1,2,\ldots,n$.
%Then $A_{\bf p}$ acts on the space of functions supported on the vertex set of ${\rm Cay}(F_{n})$ 
%by 
%\[
%A_{\bf p}f(x)=\sum_{x\sim y}p(x,y)f(y),
%\]
%where $x \sim y$ means that $y$ is adjacent to $x$.
The discrete Laplacian on ${\rm Cay}(F_{n})$ of weight 
\[
{\bf p} \in \mathcal P:=\{{\bf p}=(p_1,\ldots,p_n) \mid p_1+\cdots +p_n=1/2,\ p_i>0\ (\forall i)\}
\]
is defined by $\Delta_{{\bf p}}=I-A_{\bf p}$.

\begin{proposition} \label{totalmassfunction is constant} Let $\mu=\{\mu_{x}\}_{x\in X_{{\bf r}}}$
be the Patterson measure for $F_{n}$ on $\partial X_{{\bf r}}$.
Then 
\[
\int_{\partial X_{{\bf r}}}\Delta_{{\bf p}}j(x,\xi)^{\delta({\bf r})}d\mu_{o}(\xi)=0
\]
for every ${\bf p} \in \mathcal P$, where $\delta({\bf r})=\delta_{F_{n}}({\bf r})$.
\end{proposition}

\begin{proof} Since $F_{n}$ acts transitively on the vertices of
${\rm Cay}(F_{n})$, the $F_{n}$-invariant function $\varphi_{\mu}(x)=\int_{\partial X_{{\bf r}}}d\mu_{x}$
is constant. Hence, for every ${\bf p}$, 
\[
\Delta_{{\bf p}}\varphi_{\mu}(x)=\int_{\partial X_{{\bf r}}}\Delta_{{\bf p}}j(x,\xi)^{\delta({\bf r})}d\mu_{o}(\xi)=0.
\]
 \end{proof}
We compute $\Delta_{{\bf p}}j(x,\xi)^{s}$ and obtain the following:
if $\xi\in\partial X_{{\bf r}}$ is in the direction of $a_{i}$ or
$a_{i}^{-1}$ starting from a vertex $x\in X_{r}$ for $i=1,2,\ldots,n$,
then 
\[
\frac{\Delta_{{\bf p}}j(x,\xi)^{s}}{j(x,\xi)^{s}}=1-r_{i}^{-s}p_{i}-r_{i}^{s}\,p_{i}-2\sum_{k\neq i}r_{k}^{s}\,p_{k}=:c_{i}({\bf r},s,{\bf p}).
\]

\begin{proposition} \label{existence of zero of c} The functions
$c_{i}({\bf r},s,{\bf p})$ $(i=1,2,\ldots,n)$ of $s\in[0,\infty)$
satisfies the following properties for any fixed ${\bf r} \in \mathcal R$
and ${\bf p} \in \mathcal P$: 
\begin{enumerate}
\item $c_{i}({\bf r},0,{\bf p})=0$ and $\frac{\partial}{\partial s}c_{i}({\bf r},s,{\bf p})|_{s=0}>0$; 
\item $\frac{\partial^{2}}{\partial s^{2}}c_{i}({\bf r},s,{\bf p})<0$; 
\item $\lim_{s\to\infty}c_{i}({\bf r},s,{\bf p})=-\infty$. 
\end{enumerate}
Hence, each $c_{i}({\bf r},s,{\bf p})$ has a unique zero $s_{i}=s_{i}({\bf r},{\bf p})\neq0$,
and satisfies $\frac{\partial}{\partial s}c_{i}({\bf r},s,{\bf p})|_{s=s_{i}}<0$.
\end{proposition} \begin{proof} The second assertion in (1) follows
from the fact that 
\[
\frac{\partial}{\partial s}c_{i}({\bf r},s,{\bf p})|_{s=0}=[\log( r_i) r_i^{-s}p_i -\log(r_i) r_i^s p_i -2\sum_{j\neq i}\log(r_j) r_{j}^{s}p_{j}]\big|_{s=0}=-2\sum_{j\neq i}\log(r_j) p_{j}>0.
\]
The statement in (2) follows from 
\begin{align*}
\frac{ \partial^{2}}{\partial s^{2}}c_{i}({\bf r},s,{\bf p})
&=\frac{\partial}{\partial s}[\log( r_i) r_i^{-s}p_i -\log(r_i) r_i^s p_i -2\sum_{j\neq i}\log(r_j) r_{j}^{s}p_{j}]\\
&=-\log^2(r_i)r_i^{-s}p_i -\log^2(r_i)r_i^{s}p_i -2\sum_{j\neq i}\log^2(r_j) r_j^s p_{j}<0.
\end{align*}
%
%\[
%\frac{\partial^{2}}{\partial s^{2}}(\log(r_{i}^{-s}p_{i}+r_{i}^{s}p_{i}+2\sum_{j\neq i}r_{j}^{s}p_{j}))=-%\log^2(r_i)r_i^{-s}p_i -\log^2(r_i)r_i^{s}p_i -2\sum_{j\neq i}\log^2(r_j) r_j^s p_{j}<0.
%\]
The proofs of the remaining assertions are straightforward.
\end{proof}

\begin{lemma}\label{delta is zero} $\delta({\bf r})$ lies between $\min_{1\leq i\leq n}s_{i}({\bf r},{\bf p})$
and $\max_{1\leq i\leq n}s_{i}({\bf r},{\bf p})$. \end{lemma} \begin{proof}
Proposition \ref{totalmassfunction is constant} implies that, for
every vertex $x\in X_{\bf r}$, 
\[
\sum_{i=1}^{n}\int_{\partial X_{\bf r}^{i}(x)}c_{i}({\bf r},\delta({\bf r}),{\bf p})d\mu_{x}(\xi)=0,
\]
where $\partial X_{\bf r}^{i}(x)$ is the portion of $\partial X_{\bf r}$
whose points $\xi$ are in $a_{i}^{\pm1}$ directions from $x$. It
follows that the $c_{i}({\bf r},\delta({\bf r}),{\bf p})$ cannot have the same sign. By Proposition \ref{existence of zero of c},
we see that $c_{i}({\bf r},s,{\bf p})$ changes signs from positive to negative
at $s_{i}({\bf r},{\bf p})$, for each $i$. Therefore, $\delta({\bf r})$ must lie
between $\min_{1\leq i\leq n}s_{i}({\bf r},{\bf p})$ and $\max_{1\leq i\leq n}s_{i}({\bf r},{\bf p})$.
\end{proof}

By this lemma, if we have $s_{1}({\bf r},{\bf p})=\cdots=s_{n}({\bf r},{\bf p})\neq0$
for some weight ${\bf p} \in \mathcal P$, then this value coincides
with $\delta({\bf r})$. Hence, we consider simultaneous equations
\[
c_{1}({\bf r},s,{\bf p})=\cdots=c_{n}({\bf r},s,{\bf p})=0
\]
for a given ${\bf r}=(r_{1},\ldots,r_{n})$. First, we solve $c_{1}({\bf r},s,{\bf p})=\cdots=c_{n}({\bf r},s,{\bf p})$
as a system of equations of ${\bf p}$.

We change the variables from $({\bf r},s)$ 
to ${\bf u}=(u_{1},\ldots,u_{n})$ by $u_{i}=r_{i}^{s}$ for $i=1,\ldots,n$.
This correspondence defines a diffeomorphism 
\[
H:\mathcal R \times(0,\infty)\to(0,1)^{n}.
\]
We also set $c_{i}({\bf r},s,{\bf p})=c_{i}({\bf u},{\bf p})$
$(i=1,\ldots,n)$ (by the same notation) under this correspondence.
Namely, 
\[
c_{i}({\bf u},{\bf p})  =1-2\sum_{k\neq i}u_{k}p_{k}-u_{i}^{-1}p_{i}-u_{i}p_{i}
  =1-2\sum_{k=1}^nu_{k}p_{k}-(u_{i}^{-1}-u_{i})p_{i}.
\]

\begin{lemma}\label{newformulation}
Given ${\bf u}=(u_{1},\ldots,u_{n})\in(0,\infty)^{n}$, we consider
the  system of linear equations 
\[
c_{1}({\bf u},{\bf p})=\cdots=c_{n}({\bf u},{\bf p})
\]
for ${\bf p}=(p_1, \ldots, p_n)$ with $p_1+\cdots +p_n=1/2$
and let 
\[
D=D({\bf u}):=\sum_{j}\prod_{k\neq j}(u_{k}^{-1}-u_{k}).
\]
%\begin{enumerate}
%\item 

$(1)$ If $D\neq0$ then there exists a unique solution ${\bf p}={\bf p}({\bf u})$ given
by 
\[
p_{i}=p_{i}({\bf u})=\frac{\prod_{k\neq i}(u_{k}^{-1}-u_{k})}{2\sum_{j}\prod_{k\neq j}(u_{k}^{-1}-u_{k})}\quad(i=1,\ldots,n).
\]
The common value $\lambda=\lambda({\bf u}):=c_{1}({\bf u},{\bf p})=\cdots=c_{n}({\bf u},{\bf p})$
is given by
\begin{align*}
\lambda&=D^{-1}\left(\sum_{j}(1-u_{j})\prod_{k\neq j}(u_{k}^{-1}-u_{k})\ -\frac{1}{2}\prod_{\ell}(u_{\ell}^{-1}-u_{\ell})\right)\\
&=D^{-1}\prod_{\ell}(u_{\ell}^{-1}-u_{\ell})\left(\sum_{j}\left(\frac{u_{j}}{1+u_{j}}\right)-\frac{1}{2}\right).
%=\frac{\sum_{j}\left(\frac{1-u_{j}}{u_{j}^{-1}-u_{j}}\right)-\frac{1}{2}}{\sum_{j}\frac{1}{u_{j}^{-1}-u_{j}}}.
\end{align*}
Moreover, there exists at most one $j$ such that $u_{j}=1$, and
in that case we have that the solution is given by $p_{j}=1/2$ and $p_i=0$ for all $i \neq j$ with
$\lambda=0$. 

%\item 
$(2)$ If there exists a solution ${\bf p} \in \mathcal P$ 
(i.e., $p_{i} > 0$ for all $i$), then either $u_{i}=1$ for all $i$, $u_{i}>1$ for all $i$,
or $u_{i} <1$ for all $i$. In the first case, $D=0$ and
every ${\bf p}$ is a solution with $\lambda=0$. In the second case,
$D\neq0$ and the above formulas hold with $\lambda<0$. In the third case,
$D\neq0$ and the above formulas hold but the sign of $\lambda$ is indefinite.
%\end{enumerate}
\end{lemma}

\begin{proof}
If $D\neq0$, then existence and uniqueness of solutions follows by
verifying that $D$ is the determinant of the system of equations.
More explicitly, we can solve these equations as follows.
We first note that
$c_{1}({\bf u},{\bf p})=\cdots=c_{n}({\bf u},{\bf p})$
is equivalent to 
\begin{equation}\label{reduced-equation}
(u_{1}^{-1}-u_{1})p_{1}=(u_{2}^{-1}-u_{2})p_{2}=\dots=(u_{n}^{-1}-u_{n})p_{n}. 
\end{equation}
We set this common value as $\tau$. If $\tau \neq 0$, then we have $p_i=\tau/(u_{i}^{-1}-u_{i})$ for all $i$. 
Since $\sum_j p_j=1/2$, it follows that
$\tau \sum_j (u_{j}^{-1}-u_{j})^{-1}=1/2$. Hence,
\[
p_i=\frac{(u_{i}^{-1}-u_{i})^{-1}}{2 \sum_j (u_{j}^{-1}-u_{j})^{-1}}=
\frac{\prod_{k\neq i}(u_{k}^{-1}-u_{k})}{2\sum_{j}\prod_{k\neq j}(u_{k}^{-1}-u_{k})}.
\]
Since  $D \neq 0$, it is clear that there exists at most one $j$ with $u_{j}=1$.
If so, then $\tau=0$, $p_j=1/2$ and $p_i=0$ for all $i \neq j$, which also satisfies the above formulas for  $p_i$.
The common value $\lambda$ is obtained by substituting these solutions to any of $c_{i}({\bf u},{\bf p})$.
%The other assertions in (1) follow from (\ref{reduced-equation}).

To prove (2), suppose that there exists a solution $\bf p$ with $p_{i}>0$.
Then, according to the value $\tau$ of (\ref{reduced-equation}), we have that $u_{i}=1$, $u_{i}>1$,
or $u_{i} <1$ for all $i$ simultaneously.
%If there exists $u_{i}$ with $u_{i}=1$ then we have $u_{1}=u_{2}=\dots=u_{n}=1$
%by the arguments in (1). 
The other assertions follow from the representations of $D$ and $\lambda$. 
\end{proof}

By this lemma, the original problem to obtain $\delta({\bf r})$ is
reduced to finding a system of solutions of the equation $\lambda({\bf u})=0$
concerning ${\bf u} \in (0,1)^n$. Then by $H^{-1}{\bf u}$, we have a system
of equations for ${\bf r}$ and $s$. From this, for a given ${\bf r} \in \mathcal R$, we can obtain the exponent $s\in(0,\infty)$ which is equal to $\delta({\bf r})$.

For this purpose, we give 
another representation of $\lambda({\bf u})$ obtained in Lemma \ref{newformulation}
as follows: 
\begin{align*}
\lambda({\bf u}) & =D({\bf u})^{-1}\left(\sum_{j}(1-u_{j})\prod_{k\neq j}(u_{k}^{-1}-u_{k})\ -\frac{1}{2}\prod_{\ell}(u_{\ell}^{-1}-u_{\ell})\right)\\
 & =\frac{\prod_{i}u_{i}\cdot\left(\sum_{j}(1-u_{j})\prod_{k\neq j}(u_{k}^{-1}-u_{k})\ -\frac{1}{2}\prod_{\ell}(u_{\ell}^{-1}-u_{\ell})\right)}{\prod_{i}u_{i}\cdot D({\bf u})}\\
 & =\frac{\left(\sum_{j}(1-u_{j})u_{j}\prod_{k\neq j}(1-u_{k}^{2})\ -\frac{1}{2}\prod_{\ell}(1-u_{\ell}^{2})\right)}{\prod_{i}u_{i}\cdot D({\bf u})}\\
 & =\frac{\prod_j(1-u_{j})\left(2\sum_{j}u_{j}\prod_{k\neq j}(1+u_{k})\ -\prod_{\ell}(1+u_{\ell})\right)}{2\prod_{i}u_{i}\cdot D({\bf u})}.
% & =\frac{(-1)^{n}\prod_{i}(-1+u_{i})\,l({\bf u})
% (-1+\sum_{i}u_{i}+3\sum_{(i_{1},i_{2})}u_{i_{1}}u_{i_{2}}+5\sum_{(i_{1},i_{2},i_{3})}u_{i_{1}}u_{i_{2}}u_{i_{3}}+\cdots+(2n-1)u_{1}\cdots u_{n})
% }{2\prod_{i}u_{i}\cdot D({\bf u})},
% & =\frac{(-1)^n(\prod_{i}(-1+u_{i}))(-1+\sum_{i}u_{i}+3\sum_{(i_{1},i_{2})}u_{i_{1}}u_{i_{2}}+5\sum_{(i_{1},i_{2},i_{3})}u_{i_{1}}u_{i_{2}}u_{i_{3}}+\cdots+(2n-1)u_{1}\cdots u_{n})}{2(-\sum_{i}u_{i}+\sum_{(i_{1};i_{2})}u_{i_{1}}^{2}u_{i_{2}}-\sum_{(i_{1},i_{2};i_{3})}u_{i_{1}}^{2}u_{i_{2}}^{2}u_{i_{3}}+\cdots+(-1)^{n}\sum_{(i_{1},\ldots i_{n-1};i_{n})}u_{i_{1}}^{2}\cdots u_{i_{n-1}}^{2}u_{i_{n}})}.
\end{align*}
Here, we define 
\[
l({\bf u}):=2\sum_{j}u_{j}\prod_{k\neq j}(1+u_{k})-\prod_{\ell}(1+u_{\ell}).
\]
Then, $\lambda({\bf u})=0$ is equivalent to $l({\bf u})=0$ for ${\bf u} \in (0,1)^n$.

\begin{proposition}\label{delta}
For every ${\bf r} \in \mathcal R$, $\delta({\bf r})$  is the unique $s > 0$
such that ${\bf u}=H({\bf r},s)$ satisfies $l({\bf u})=0$.
\end{proposition}

\begin{proof}
We will find $s > 0$ such that $l(H({\bf r},s))=0$.
For a fixed ${\bf r}$, it is easy to see that $l \circ H({\bf r},s))$ is a strictly decreasing continuous function
such that $\lim_{s \to 0} l \circ H({\bf r},s))>0$ and $\lim_{s \to \infty} l \circ H({\bf r},s))=-1$. Hence, such an $s$ uniquely exists. That $s=\delta({\bf r})$ follows from Lemma \ref{delta is zero}.
\end{proof}

\begin{proof}[Proof of Theorem \ref{main1}]
By expanding $l({\bf u})$, we have
\[
l({\bf u})=-1+\sum_{i}u_{i}+3\sum_{(i_{1},i_{2})}u_{i_{1}}u_{i_{2}}+5\sum_{(i_{1},i_{2},i_{3})}u_{i_{1}}u_{i_{2}}u_{i_{3}}+\cdots+(2n-1)u_{1}\cdots u_{n}.
\] 
Then, the statement follows from Proposition \ref{delta}.
\end{proof}

\section{$\lambda_0$ in terms of $\delta$ on ${\rm Cay}(F_n)$}
In this section, we will prove Theorem \ref{main2}. To this end,
we consider the maximal value of $\lambda({\bf u})$ for ${\bf u} \in (0,1)^n$ under a constraint condition
${\bf p}({\bf u})={\bf p}_0$ for some fixed ${\bf p}_0=(p_1,\ldots, p_n) \in \mathcal P$. 
We note that the condition 
${\bf p}({\bf u})={\bf p}_0$ is equivalent to
$c_1({\bf u}, {\bf p}_0)=\cdots =c_n({\bf u}, {\bf p}_0)$, which is further equivalent to
\[
(u_1^{-1}-u_1) p_1=(u_2^{-1}-u_2) p_2=\cdots =(u_n^{-1}-u_n) p_n 
\]
for ${\bf u}=(u_1,\ldots ,u_n)$ by (\ref{reduced-equation}). 

Putting the common value of these equations as $\tau \in (0,\infty)$,
we can solve $u_i \in (0,1)$ for each $i$ as
\begin{equation}\label{u-function}
u_i=u_i(\tau)=\frac{1}{2}(\sqrt{\tau^2 p_i^{-2}+4}-\tau p_i^{-1}).
\end{equation}
Then, we have a smooth curve $\gamma_{{\bf p}_0}(\tau):=(u_1(\tau),\ldots,u_n(\tau))$ $(0<\tau<\infty)$ in $(0,1)^n$ such that
\[
\{\gamma_{{\bf p}_0}(\tau) \mid \tau \in (0,\infty)\}=\{{\bf u} \mid {\bf p}({\bf u})={\bf p}_0\}.
\]
Moreover,
$\lim_{\tau \to 0} \gamma_{{\bf p}_0}(\tau)=(1,\ldots,1)$ and $\lim_{\tau \to \infty} \gamma_{{\bf p}_0}(\tau)=(0,\ldots,0)$.

\begin{proposition}\label{derivative}
For every ${\bf p}_0=(p_1,\ldots, p_n) \in \mathcal P$, the function $\lambda \circ \gamma_{{\bf p}_0}(\tau)$ on $(0,\infty)$ takes the unique maximum 
at $\tau_0$ where the derivative
\[
(\lambda \circ \gamma_{{\bf p}_0})'(\tau)=-\sum_i \frac{\tau}{\sqrt{\tau^2+4p_i^2}}+(n-1)
\] 
vanishes. Moreover, $\lambda \circ \gamma_{{\bf p}_0}(\tau_0)>0$.
\end{proposition}

\begin{proof}
By substituting (\ref{u-function}) to $\lambda({\bf u})=c_i({\bf u},{\bf p}_0)$, we have that
\begin{align*}
\lambda \circ \gamma_{{\bf p}_0}(\tau)
&=1-2 \sum_i\frac{1}{2}(\sqrt{\tau^2 p_i^{-2}+4}-\tau p_i^{-1})p_i-\tau\\
&=1-\sum_i\sqrt{\tau^2 +4p_i^2}+(n-1)\tau.
\end{align*}
Then $\lim_{\tau \to 0}\lambda \circ \gamma_{{\bf p}_0}(\tau)=0$ and 
$\lim_{\tau \to \infty}\lambda \circ \gamma_{{\bf p}_0}(\tau)=-\infty$.

Moreover, the derivative of $\lambda \circ \gamma_{{\bf p}_0}(\tau)$ is
\[
(\lambda \circ \gamma_{{\bf p}_0})'(\tau)
=-\sum_i \frac{\tau}{\sqrt{\tau^2+4p_i^2}}+(n-1).
\]
This is a strictly decreasing continuous function from a positive $\lim_{\tau \to 0}(\lambda \circ \gamma_{{\bf p}_0})'(\tau)=n-1$
to a negative $\lim_{\tau \to \infty}(\lambda \circ \gamma_{{\bf p}_0})'(\tau)=-1$.
The statement then follows easily.
\end{proof}

The following claim shows the way of choosing ${\bf r} \in \mathcal R$ corresponding to ${\bf p} \in \mathcal P$.

\begin{lemma}\label{maximum2} 
For every ${\bf p}_0 \in \mathcal P$, 
assume that the function $\lambda \circ \gamma_{{\bf p}_0}(\tau)$ for $\tau \in (0,\infty)$ takes the unique maximum
at $\tau_0$. Then, there exists a unique ${\bf r}_0 \in \mathcal R$
such that $H({\bf r}_0,\delta({\bf r}_0)/2)=\gamma_{{\bf p}_0}(\tau_0)$.
\end{lemma}

\begin{proof}
Set $\gamma_{{\bf p}_0}(\tau_0)={\bf u}_0=(u_1,\ldots,u_n)$.
Then we have
\[
(u_1^{-1}-u_1) p_1=(u_2^{-1}-u_2) p_2=\cdots =(u_n^{-1}-u_n) p_n =\tau_0
\]
for ${\bf p}_{0}=(p_1,\ldots,p_n)$. 
Since  $\tau_0^{-1}p_i=u_i/(1-u_i^2)$ for all $i$, which follows from the above equations, we have that
\begin{align*}
(\lambda \circ \gamma_{{\bf p}_0})'(\tau_0)
&=-\sum_i \frac{1}{\sqrt{1+4(\tau_0^{-1} p_i)^2}}+(n-1)\\
&=-\sum_i \frac{1-u_i^2}{\sqrt{(1-u_i^2)^2+4u_i^2}}+(n-1)\\
&=-\sum_i \frac{1-u_i^2}{1+u_i^2}+(n-1)\\
&=\frac{-\sum_i(1-u_i^2) \prod_{k \neq i}(1+u_k^2)+(n-1) \prod_k(1+u_k^2)}{\prod_k(1+u_k^2)}.
\end{align*}
Since $(\lambda \circ \gamma_{{\bf p}_0})'(\tau_0)=0$, we have that the numerator 
\begin{align*}
&\quad -\sum_i(1-u_i^2) \prod_{k \neq i}(1+u_k^2)+(n-1) \prod_k(1+u_k^2)\\
&=2 \sum_i u_i \prod_{k \neq i}(1+u_k^2)-\prod_k(1+u_k^2)\\ 
&\quad +\left(- \sum_i  \prod_{k \neq i}(1+u_k^2) - \sum_i u_i^2 \prod_{k \neq i}(1+u_k^2) +n\prod_k(1+u_k^2) \right) \\
&=2 \sum_i u_i^2 \prod_{k \neq i}(1+u_k^2)-\prod_k(1+u_k^2).
%-\sum_i  \prod_{k }(1+u_k) +n\prod_k(1+u_k). 
\end{align*}
is equal to zero. We define $({\bf r}_0,s_0)$ to be $H^{-1}({\bf u}_0)$ and set ${\bf r}_0=(r_1,\ldots,r_n)$.
By the definition of the function $l$, we have that 
$l \circ H({\bf r}_0,2s_0)=0$. This implies that $2s_0=\delta({\bf r}_0)$ by Proposition \ref{delta}.
Hence, ${\bf u}_0=H({\bf r}_0,\delta({\bf r}_0)/2)$.
\end{proof}

\begin{remark}
The above proof also implies that if ${\bf u}_0$ is given by $H({\bf r}_0,\delta({\bf r}_0))$ for 
any ${\bf r}_0 \in \mathcal R$, then $\lambda({\bf u})$ takes the maximum at ${\bf u}_0$ under the constraint condition
${\bf p}({\bf u})={\bf p}_0:={\bf p}({\bf u}_0)$.
The fact that ${\bf u}_0$ is the critical point of
$\lambda({\bf u})$ is also verified by the method of Lagrange multiplier
without using $\lambda \circ \gamma_{{\bf p}_0}(\tau)$.
We note that since ${\bf p}=(p_{1},\ldots,p_{n})$
satisfies $p_{1}+\cdots+p_{n}=1/2$, the constraint condition can be determined
only by $p_{1},\ldots,p_{n-1}$. If $\lambda({\bf u})$ attains a constraint local maximum or minimum
at ${\bf u}_0$, then
${\bf u}_0=(u_{1},\ldots,u_{n})$ must satisfy 
\[
{\rm det}\left[\begin{array}{rrrr}
\frac{\partial p_{1}}{\partial u_{1}}({\bf u}) & \cdots & \frac{\partial p_{n-1}}{\partial u_{1}}({\bf u}) & \frac{\partial\lambda}{\partial u_{1}}({\bf u})\\
\frac{\partial p_{1}}{\partial u_{2}}({\bf u}) & \cdots & \frac{\partial p_{n-1}}{\partial u_{2}}({\bf u}) & \frac{\partial\lambda}{\partial u_{2}}({\bf u})\\
\vdots\quad & \  & \vdots\qquad & \vdots\quad\\
\frac{\partial p_{1}}{\partial u_{n}}({\bf u}) & \cdots & \frac{\partial p_{n-1}}{\partial u_{n}}({\bf u}) & \frac{\partial\lambda}{\partial u_{n}}({\bf u})
\end{array}\right]=0.
\]
By Mathematica, we can check that this is equivalent to $l(u_{1}^{2},\ldots,u_{n}^{2})=0$.
\end{remark}

If we start from the edge length parameter ${\bf r} \in \mathcal R$, 
our main result in this section can be alternatively expressed as follows. 
This will be discussed again in the next section.

\begin{theorem}\label{main2-b}
For any ${\bf r}_0 \in {\mathcal R}$, the bottom of the spectrum  $\lambda_0({\bf p}_0)$
of the Laplacian $\Delta_{{\bf p}_0}$ for ${\bf p}_0={\bf p}\circ H({\bf r}_0,\delta({\bf r}_0)/2)$
on $X_{{\bf r}_0}$ coincides with $\lambda \circ H({\bf r}_0,\delta({\bf r}_0)/2)$.
\end{theorem}

\begin{proof} 
%Fix ${\bf r}_{0}$ and show that 
%\[
%\rho({\bf p}_{0})=1-\lambda({\bf r}_{0},\delta({\bf r}_{0})/2),
%\]
%where we set ${\bf p}_{0}={\bf p}({\bf r}_{0},\delta({\bf r}_{0})/2)$.
It is well known that 
(see e.g. Lemmas 7.2 and 7.6 in Woess \cite{MR1743100}) the
spectral radius $\rho({\bf p}_{0})$ of the Markov chain determined
by $A_{{\bf p}_{0}}$ is the minimum of eigenvalues for positive eigenfunctions $h$ on ${\rm Cay}(F_n)$. 
Since $\Delta_{{\bf p}_0}=I-A_{{\bf p}_0}$ and $\lambda_0({\bf p}_0)=1-\rho({\bf p}_0)$, we see that
%\begin{equation}
%\rho({\bf p}_{0})=\min\{t\geq0\mid\exists h\geq0,\ \Delta_{{\bf p}_{0}}h=(1-t)h\}.\label{spectralradius2}
%\end{equation}
\[
\lambda_0({\bf p}_{0})=\max\{\lambda \mid\exists h\geq0,\ \Delta_{{\bf p}_{0}}h=\lambda h\}.
%\label{spectralradius2}
\]
Let $H({\bf r}_0,\delta({\bf r}_0)/2)={\bf u}_0$.
By the definition of the function $\lambda$, 
the positive function $h(x)=j_{{\bf r}_{0}}(x,\xi)^{\delta({\bf r}_0)/2}$ for any $\xi \in \partial X_{{\bf r}_{0}}$
%\[
%h(x)=\int_{\partial X_{{\bf r}_{0}}}j_{{\bf r}_{0}}(x,\xi)^{\delta({\bf r}_{0})/2}d\mu_{o}\geq0
%\]
satisfies $\Delta_{{\bf p}_{0}}h=\lambda({\bf u}_{0})h$.
%where $\mu_{o}$ is any conformal measure on $\partial X_{{\bf r}_{0}}$.
From this, we have $\lambda({\bf u}_{0}) \leq \lambda_0({\bf p}_{0})$.
Hence, the problem is to show the converse inequality.

By Lemma \ref{newformulation}, we
see that if some ${\bf u} \neq (1,\ldots,1)$ satisfies the simultaneous equations
$\lambda=c_{i}({\bf u},{\bf p})$ $(i=1,\ldots ,n)$ for given ${\bf p}={\bf p}_{0}$
and $\lambda=\lambda_{0}$, then ${\bf p}_{0}$ and $\lambda_{0}$ are represented as
${\bf p}({\bf u})$ and $\lambda({\bf u})$, respectively.
Theorem \ref{final2} below asserts that for $\lambda_{0}=\lambda_0({\bf p}_{0})$
there exists some ${\bf u}_1 \in (0,1)^n$ that satisfies these equations.
By the fact mentioned above, $\lambda_0({\bf p}_{0})$ is represented as
$\lambda({\bf u}_1)$ by using this ${\bf u}_1$,
which also satisfies  the condition ${\bf p}({\bf u}_1)={\bf p}_{0}$.

We consider the function $\lambda({\bf u})$ of variables
${\bf u} \in (0,1)^n$ under the constraint ${\bf p}({\bf u})={\bf p}_0$.
Then, by the proof of Lemma \ref{maximum2} (see also the remark after the proof), 
we have $\lambda({\bf u})\leq\lambda({\bf u}_{0})$.
This yields the desired inequality 
$\lambda_0({\bf p}_{0})\leq \lambda({\bf u}_{0})$, which completes the proof. 
\end{proof}

The arguments above imply formula  (\ref{Woess}).
We also note that if we assume (\ref{Woess}), then we can prove Theorem \ref{main2-b}
without showing Theorem \ref{final2}. To prove  (\ref{Woess}) we proceed as follows. By Theorem \ref{main2-b}, $\rho({\bf p}_0)$ is given by $1-\lambda({\bf u}_0)$, where
$\lambda({\bf u}_0)$ is the maximal value of $\lambda({\bf u})$ under the constraint
condition ${\bf p}({\bf u})={\bf p}_0$
by Lemma \ref{maximum2}. Proposition \ref{derivative} implies that this constraint maximum coincides with
$\max_{\tau \in (0,\infty)} \lambda \circ \gamma_{{\bf p}_0}(\tau)$.
%Recall that $\rho({\bf p})=1-\max\, \{ \lambda \mid \exists h\ge0 \textrm{ such that }  \Delta_{\bf p} h=\lambda h\}$. 
%By Theorem \ref{final2},
%it is sufficient to consider simultaneous equations $c_{1}({\bf u},{\bf p})=\cdots=c_{n}({\bf u},{\bf p})$. 
%By the definition of $\lambda$ and $\gamma_{\bf p}$, 
Hence, we have 
\[
\rho({\bf p}_0)=1-\max_{\tau \in (0,\infty)} \lambda(\gamma_{{\bf p}_0}(\tau))
=\min_{\tau \in (0,\infty)} \frac{(1-\lambda(\gamma_{{\bf p}_0}(\tau)))\tau^{-1}}{\tau^{-1}},
\]
and by a short calculation using the formula 
\[
\lambda(\gamma_{{\bf p}_0}(\tau))=1-\sum_i\sqrt{\tau^2 +4p_i^2}+(n-1)\tau,
\]
the desired formula  (\ref{Woess}) follows.

We construct a solution ${\bf u} \in (0,1)^n$ of the equations $\lambda=c_i({\bf u},{\bf p})$ $(i=1,\ldots,n)$
for given ${\bf p} \in \mathcal P$ and $\lambda \geq 0$ in the following way.

\begin{theorem}\label{final2}
For a given ${\bf p} \in \mathcal P$, if $0 \leq \lambda \leq \lambda_0({\bf p})$, then
the simultaneous equations $\lambda=c_i({\bf u},{\bf p})$ $(i=1,\ldots,n)$ have
a solution ${\bf u}$ in $(0,1)^n$.
\end{theorem}

\begin{proof} 
By $\lambda_0({\bf p})=1-\rho({\bf p})$, the condition 
$0 \leq \lambda \leq \lambda_0({\bf p})$ is equivalent to $\rho({\bf p}) \leq t \leq 1$
for $t:=1-\lambda$.
Since 
\[
\rho({\bf p})=\Vert A_{\bf p}\Vert=\lim_{m\rightarrow\infty}\Vert A_{\bf p}^{m}\Vert^{1/m}>0
\]
for the operator norm $\Vert A_{\bf p}\Vert$ of the transition matrix $A_{\bf p}=(p(x,y))_{x,y}$
of the Markov chain
acting on $L^2({\rm Cay}(F_n))$,
we have that the Green function 
\[
G_{t}(x,w):=\sum_{m=0}^{\infty}p^{m}(x,w)t^{-m}
\]
converges for every $t>\rho({\bf p})$  for all vertices $x,w \in {\rm Cay}(F_{n})$. Here, $p^m(x,y)$ denotes
the entry of $A^m_{\bf p}$. In fact, 
it is known that also $G_{t}<\infty$ if $t=\rho({\bf p})$ because the random walk determined by $A_{\bf p}$ on ${\rm Cay}(F_n)$ is
$\rho({\bf p})$-transient (see Theorem 7.8 in [Woe00]).

We denote by $f^{(m)}(e,g)$ the probability
that the random walk, starting at the group identity $e$, hits the
element $g$ after $m$ steps for the first time. Since $f^{(m)}(e,g)\le p^{m}(e,g)$
and $G_{t}<\infty$, we can define 
\[
u_{i}:=\sum_{m=0}^{\infty}f^{(m)}(e,a_{i})t^{-m} >0\quad(i=1,\ldots,n).
\]
Note that $f^{(0)}(e,a_{i})=0$. 

We first prove that ${\bf u}=(u_{1},\ldots,u_{n})$
defined as above satisfies
\[
1-t=c_i({\bf u},{\bf p}) \quad (i=1,\ldots,n).
\]
We write $u_i c_i({\bf u},{\bf p})$ for $i=1,\ldots,n$
as 
\[
p_{i}+u_{i}(p_{i}u_{i}+2\sum_{k\neq i}p_{k}u_{k})=t(t^{-1}p_{i}+u_{i}(t^{-1}p_{i}u_{i}+2t^{-1}\sum_{k\neq i}p_{k}u_{k})).
\]
Then, it suffices to show that 
\[
u_{i}=t^{-1}p_{i}+u_{i}(t^{-1}p_{i}u_{i}+2t^{-1}\sum_{k\neq i}p_{k}u_{k}).
\]
Decomposition of the event of ever hitting $a_{i}$ gives 
\begin{align*}
u_{i}=\sum_{m=1}^{\infty}f^{(m)}(e,a_{i})t^{-m} & =p(e,a_{i})t^{-1}+p(e,a_{i}^{-1})t^{-1}\sum_{m=1}^{\infty}f^{(m)}(a_{i}^{-1},e)t^{-m}\sum_{m=1}^{\infty}f^{(m)}(e,a_{i})t^{-m}\\
 & +\sum_{k\neq i}p(e,a_{k})t^{-1}\sum_{m=1}^{\infty}f^{(m)}(a_{k},e)t^{-m}\sum_{m=1}^{\infty}f^{(m)}(e,a_{i})t^{-m}\\
 & +\sum_{k\neq i}p(e,a_{k}^{-1})t^{-1}\sum_{m=1}^{\infty}f^{(m)}(a_{k}^{-1},e)t^{-m}\sum_{m=1}^{\infty}f^{(m)}(e,a_{i})t^{-m}.
\end{align*}
It follows that 
\[
u_{i}=p_{i}t^{-1}+p_{i}t^{-1}u_{i}u_{i}+2\sum_{k\neq i}p_{k}t^{-1}u_{k}u_{i}
\]
for each $i$. 

Finally, we verify that $\bf u$ is in $(0,1)^n$. In the case when $t=1$, we have that
\[
u_{i}:=\sum_{m=0}^{\infty}f^{(m)}(e,a_{i})  <1 \quad(i=1,\ldots,n)
\]
since the random walk is transient. If $t<1$, then we consider the original equations
$\lambda=c_i({\bf u},{\bf p})$ $(i=1,\ldots,n)$ for $\lambda>0$. By Lemma \ref{newformulation} (2),
we see that $\bf u$ satisfies $u_i<1$ for all $i$. Thus, we have ${\bf u} \in (0,1)^n$ in any case.
\end{proof}

\begin{remark}
Ledrappier \cite[Lemma 2.2]{MR1832436} considered a solution of equivalent equations 
to the above in the case when $t=1$. 
%but not necessarily symmetric with $p_{i}=p_{-i}$.
\end{remark}

By Lemma \ref{maximum2} and Theorem \ref{main2-b},
we obtain the theorem mentioned in the introduction.

\begin{proof}[Proof of Theorem \ref{main2}]
For a given ${\bf p} \in \mathcal P$, choose ${\bf r} \in \mathcal R$ as in Lemma \ref{maximum2}.
Then, Theorem \ref{main2-b} yields the assertion.
\end{proof}

\section{Generalization of the cogrowth formula}
%{A formula for the bottom of spectra of the Laplacian}

We investigate the relationship between the Poincar\'e exponent and the bottom of the spectrum  of the Laplacian for a
subgroup $G \subset F_{n}$. 
For an edge length parameter ${\bf r} \in \mathcal R$,  we denote by 
$X_{\bf r}$  the Cayley graph ${\rm Cay}(F_{n})$ with the distance $d_{\bf r}$ as before.
Since $G$ acts on $X_{\bf r}$
isometrically, discontinuously and freely, we obtain the quotient graph $G \backslash {\rm Cay}(F_n)$
endowed with the metric induced by $d_{\bf r}$. 
We use an appropriate weight ${\bf p} \in \mathcal P$ to consider the Laplacian $\Delta_{\bf p}$ on $G \backslash {\rm Cay}(F_n)$.
By the facts shown in the previous sections, we see that the weight ${\bf p}$ can be given not only in terms of 
${\bf r}$ but also depending
on the dimension $s=\delta_G({\bf r})$ of a subgroup $G \subset F_n$.

We will prove Theorem \ref{main3}
by dividing it into two cases according to formula (\ref{eps}).
The first case follows from  the following claim, which is the main part of the cogrowth formula.

\begin{theorem}\label{firstline}
For any subgroup $G\subset F_{n}$ and any ${\bf r}\in {\mathcal R}$,
if $\delta_{G}({\bf r})> \delta({\bf r})/2$, then
the bottom of the spectrum  $\lambda_0^{G}({\bf p})$ of the Laplacian $\Delta_{{\bf p}}$
for ${\bf p}={\bf p}\circ H({\bf r},\delta_{G}({\bf r}))$
on the quotient graph $G \backslash {\rm Cay}(F_n)$ coincides with
$\lambda \circ H({\bf r},\delta_{G}({\bf r}))$. 
\end{theorem}

\begin{proof}
Let  $\mu=\{\mu_{x}\}_{x\in X_{{\bf r}}}$ be a  Patterson measure for $G$. Consider the positive $G$-invariant 
total mass function
$$
\varphi_\mu(x)=\int_{\partial X_{\bf r}} d\mu_x=\int_{\partial X_{\bf r}} j(x,\xi)^{\delta_G({\bf r})} d\mu_o(\xi)
\quad (x \in X_{\bf r}).
$$
For ${\bf p}={\bf p}\circ H({\bf r},\delta_G({\bf r}))$, we have
\begin{align*}
\Delta_{\bf p} \varphi_\mu(x)&=\int_{\partial X_{\bf r}} \Delta_{\bf p} j(x,\xi)^{\delta_G({\bf r})} d\mu_o(\xi)\\
&=\int_{\partial X_{\bf r}} \lambda \circ H({\bf r},\delta_G({\bf r})) j(x,\xi)^{\delta_G({\bf r})} d\mu_o(\xi)
=\lambda \circ H({\bf r},\delta_G({\bf r}))\varphi_\mu(x).
\end{align*}
Therefore, $\varphi_\mu$ descends to a positive eigenfunction of $\Delta_{\bf p}$ on $G \backslash X_{\bf r}$
with the eigenvalue $\lambda \circ H({\bf r},\delta_G({\bf r}))$. Since $\lambda_0^G({\bf p})$ 
is known to be the maximum of such eigenvalues,
we conclude that $\lambda \circ H({\bf r},\delta_G({\bf r})) \leq \lambda_0^G({\bf p})$.

For the converse inequality, we first assume that $G$ is finitely generated and
show that $\varphi_\mu \in L^2(G \backslash {\rm Cay}(F_n))$.
Since $G$ is convex cocompact, the quotient graph $G \backslash {\rm Cay}(F_n)$ consists of a finite core graph $C_G$ to which
a finite number of rooted regular trees $(T_i,\hat x_i)$ $(i=1,\ldots,m)$ of valency $2n$ (the valency at $\hat x_i$
is 1) are  attached. Let $(\widetilde T_i,x_i)$ be a connected component of the inverse image of $(T_i,\hat x_i)$ under the 
quotient map $X_{\bf r} \to G \backslash {\rm Cay}(F_n)$. We note that the restriction of the quotient map to $(\widetilde T_i,x_i)$
is an isometry onto $(T_i,\hat x_i)$.
It suffices to show that $\varphi_\mu$ is square integrable on each $\widetilde T_i$.

We estimate $j(x,\xi)$ for $x \in \widetilde T_i$ by representing it as
\[
j(x,\xi)=\exp(-b_\xi(x))=\exp\{2d_{\bf r}(o,y_i)-d_{\bf r}(o,x)\},
\]
where $y_i$ is the nearest point from $x$ to the geodesic ray $[o,\xi)$. 
We may assume that the projection of the base point $o$ is in $C_G$.
If $\xi$ is a limit point of $G$, 
then the projection of the geodesic ray $[o,\xi)$ is in $C_G$, from which we see that
$y_i$ is on the geodesic segment $[o,x_i]$. In particular, there is some $C_i>0$ such that
$\exp(2d_{\bf r}(o,y_i) )\leq C_i$ 
for every limit point $\xi\in \partial X_{\bf r}$ and for every $x \in \widetilde T_i$. 

The above estimate of $j(x,\xi)$ implies that 
\[
\varphi_\mu(x) \leq C_i e^{-\delta_G({\bf r}) d_{\bf r}(x_i,x)} \quad (x \in \widetilde{T_i})
\]
for each $i$.
Then, we obtain that
\[
\sum_{x \in \widetilde T_i} \varphi_\mu(x)^2 \leq C_i^2 \lim_{R \to \infty} \int_0^R e^{-2\delta_G({\bf r})t} d N(t) \quad\quad (i=1,\ldots,m)
\]
for $N(t):=\#\{x \in \widetilde T_i \mid d_{\bf r}(x_i,x) \leq t\}$. 

We choose some $\varepsilon >0$ such that $2\delta_G({\bf r}) \geq \delta({\bf r})+2\varepsilon$.
Since $N(t) \leq D e^{(\delta({\bf r})+\varepsilon) t}$ for some constant $D>0$, we see that
\begin{align*}
\int_0^R e^{-2\delta_G({\bf r})t} d N(t)&=e^{-2\delta_G({\bf r})R}N(R)+2\delta_G({\bf r}) \int_0^R e^{-2\delta_G({\bf r})t}N(t) dt\\
&\leq D \left(e^{-\varepsilon R}+2\delta_G({\bf r}) \int_0^R e^{-\varepsilon t}dt \right),
\end{align*}
which has a finite limit as $R \to \infty$. 
This implies that $\varphi_\mu$ is square integrable on $G \backslash {\rm Cay}(F_n)$,
and hence the eigenvalue $\lambda \circ H({\bf r},\delta_G({\bf r}))$ for $\Delta_{\bf p}$
is not less than $\lambda_0^G({\bf p})$.
Thus, we obtain that $\lambda \circ H({\bf r},\delta_G({\bf r})) = \lambda_0^G({\bf p})$ 
for a finitely generated subgroup $G \subset F_n$.

For an infinitely generated subgroup $G$, we choose an exhaustion by a sequence of finitely generated subgroups $G_k$
such that
\[
G_1 \subset G_2 \subset \cdots \subset \bigcup_k G_k=G. 
\] 
In this case, clearly $\delta_{G_1}({\bf r}) \leq \delta_{G_2}({\bf r}) \leq \cdots \leq \delta_{G_k}({\bf r}) \leq \delta_{G}({\bf r})$,
from which we can verify that
$\lim_{k \to \infty}\delta_{G_k}({\bf r}) = \delta_{G}({\bf r})$.
Indeed, we take the Patterson measure $\mu_k$ for $G_k$
with the normalization $\mu_k(\partial X_{\bf r})=1$. Then, $\{\mu_k\}$ has a subsequence that converges to
a probability measure $\mu$ on $\partial X_{\bf r}$ in the weak-$\ast$ sense. 
Note that $\delta:=\lim_{k \to \infty} \delta_{G_k}({\bf r})$ exists, which is bounded from above by $\delta_G({\bf r})$.
It is easy to see that $\mu$ is a $G$-invariant conformal measure
of dimension $\delta$. Since the dimension of any $G$-invariant conformal measure is not less than $\delta_G({\bf r})$ (see [Coo93]),
we have that $\delta \geq \delta_G({\bf r})$. Hence $\delta=\delta_G({\bf r})$.

Since $\lambda \circ H({\bf r}, \cdot)$ is continuous, we have that 
$\lim_{k \to \infty}\lambda \circ H({\bf r},\delta_{G_k}({\bf r}))=\lambda \circ H({\bf r},\delta_{G}({\bf r}))$.
Similarly, ${\bf p}_k:={\bf p}\circ H({\bf r},\delta_{G_k}({\bf r}))$ converges to ${\bf p}={\bf p}\circ H({\bf r},\delta_{G}({\bf r}))$
by the continuity of ${\bf p}\circ H({\bf r},\cdot)$. Moreover, if the weights ${\bf p}_k$ 
of the Laplacian on the graph $G \backslash X_{\bf r}$ converge to ${\bf p}$, then the bottom of the spectra 
$\lambda_{0}^G({\bf p}_k)$ converge to $\lambda_{0}^G({\bf p})$ as $k \to \infty$.
Indeed, for the inner product $\langle \cdot,\cdot \rangle$ on $L^2(G \backslash {\rm Cay}(F_n))$,
we have that $\langle \Delta_{{\bf p}_k}f,f \rangle$ converges to $\langle \Delta_{{\bf p}}f,f \rangle$
uniformly for all $f \in L^2(G \backslash {\rm Cay}(F_n))$ with $\langle f,f \rangle=1$.
On the other hand, by lifting positive eigenfunctions on $G \backslash {\rm Cay}(F_n)$ to $G_n \backslash {\rm Cay}(F_n)$,
we easily see that $\lambda_0^{G_k}({\bf p}_k) \geq \lambda_{0}^G({\bf p}_k)$.
Since $\lambda_0^{G_k}({\bf p}_k)=\lambda \circ H({\bf r},\delta_{G_k}({\bf r}))$ 
by the above arguments for finitely generated subgroups $G_k$,
we conclude that $\lambda \circ H({\bf r},\delta_{G}({\bf r})) \geq \lambda_{0}^G({\bf p})$.
\end{proof}

On the other hand, by Theorem \ref{main2-b}
obtained in the previous section, 
we can say that the proper weight of the Laplacian $\Delta_{\bf p}$ on ${\rm Cay}(F_n)$ 
in the case of $G=\{\rm id\}$ is
${\bf p}={\bf p}\circ H({\bf r},\delta({\bf r})/2)$. In the next theorem,
we show that this result can be generalized for any $G$ with $\delta_{G}({\bf r}) \leq \delta({\bf r})/2$.

\begin{theorem}\label{secondline}
For any subgroup $G\subset F_{n}$ and any ${\bf r}\in {\mathcal R}$,
if $\delta_{G}({\bf r}) \leq \delta({\bf r})/2$, then
the bottom of the spectrum $\lambda_0^{G}({\bf p})$ of the Laplacian $\Delta_{{\bf p}}$
for ${\bf p}={\bf p}\circ H({\bf r},\delta({\bf r})/2)$
on the quotient graph $G \backslash {\rm Cay}(F_n)$ coincides with
$\lambda \circ H({\bf r},\delta({\bf r})/2)$. 
\end{theorem}

\begin{proof}
We take a $G$-invariant conformal measure $\mu=\{\mu_{x}\}_{x\in X_{{\bf r}}}$ of dimension $\delta({\bf r})/2$,
which is not less than $\delta_G({\bf r})$ by assumption. In the case where $\delta({\bf r})/2=\delta_G({\bf r})$,
we just take a Patterson measure $\mu$ for $G$ by Theorem \ref{patterson}.
In the case where $\delta({\bf r})/2>\delta_G({\bf r})$,
the existence of such a measure $\mu$ can be seen as follows.
We consider the sum of weighted Dirac masses 
\[
\mu_{x,y}=\frac{1}{\sum_{g \in G}e^{-s d_{\bf r}(x,gy)}} \sum_{g \in G} e^{-s d_{\bf r}(x,gy)}{\bf 1}_{gy}
\]
for any vertices $x,y \in X_{\bf r}$ and 
for $s=\delta({\bf r})/2$. Note that the Poincar\'e series $\sum_{g \in G}e^{-s d_{\bf r}(x,gy)}$
converges if $s>\delta_G({\bf r})$.
Since $G$ is not cocompact, we can choose
a sequence $y_k \in X_{\bf r}$ within a fundamental domain of $G$ that converges to a point at infinity.
Then, a subsequence of $\{\mu_{x,y_k}\}$ converges to some $G$-invariant conformal measure $\{\mu_x\}$
of dimension $s$
in the weak-$\ast$ sense. This is a modification of the construction of ending measures for Kleinian groups
by Anderson, Falk and Tukia \cite{MR2305045}.

We consider the positive $G$-invariant 
total mass function
$$
\varphi_\mu(x)=\int_{\partial X_{\bf r}} d\mu_x=\int_{\partial X_{\bf r}} j(x,\xi)^{\delta({\bf r})/2} d\mu_o(\xi)
\quad (x \in X_{\bf r}).
$$
For ${\bf p}={\bf p} \circ H({\bf r},\delta({\bf r})/2)$, this satisfies 
$\Delta_{\bf p} \varphi_\mu=\lambda \circ H({\bf r},\delta({\bf r})/2)\varphi_\mu$.
Thus, we obtain a positive eigenfunction function for $\Delta_{\bf p}$ on $G \backslash {\rm Cay}(F_n)$
with eigenvalue $\lambda \circ H({\bf r},\delta({\bf r})/2)$. This implies that
$\lambda_0^G({\bf p}) \geq \lambda \circ H({\bf r},\delta({\bf r})/2)$.

On the other hand, $\lambda_0({\bf p}) = \lambda \circ H({\bf r},\delta({\bf r})/2)$ by Theorem \ref{main2}.
Since any positive eigenfunction for $\Delta_{\bf p}$ on $G \backslash {\rm Cay}(F_n)$ can be lifted to ${\rm Cay}(F_n)$,
we see that $\lambda_0({\bf p}) \geq \lambda_0^G({\bf p})$. This concludes that
$\lambda_0^G({\bf p}) = \lambda \circ H({\bf r},\delta({\bf r})/2)$.
\end{proof}

\begin{proof}[Proof of Theorem \ref{main3}]
By Theorems \ref{firstline} and \ref{secondline} with the definition of the 
appropriate weight ${\bf p}_{*}({\bf r},s)$ for the Laplacian, we immediately have
the result.
\end{proof}

We note that Theorem \ref{main3} for the case of ${\bf r}=(1/2n,\ldots,1/2n)$
implies the original Grigorchuk cogrowth formula. In other words, Theorem \ref{cogrowth} can be
obtained as a corollary to Theorem \ref{main3}.


\begin{thebibliography}{AFT07}

\bibitem[AO76]{MR0442698}
C.~A. Akemann and P.~A. Ostrand, \emph{Computing norms in group
  {$C\sp*$}-algebras}, Amer. J. Math. \textbf{98} (1976), no.~4, 1015--1047.
  \MR{0442698}


\bibitem[AFT07]{MR2305045}
J. W. Anderson, K. Falk and P. Tukia, 
\emph{Conformal measures associated to ends of hyperbolic $n$
-manifolds}, Q. J. Math. 58 (2007), 1--15. \MR{2305045}

\bibitem[Bro85]{MR783536}
R.~Brooks, \emph{The bottom of the spectrum of a {R}iemannian covering}, J.
  Reine Angew. Math. \textbf{357} (1985), 101--114. \MR{783536}

\bibitem[Coo93]{MR1214072}
M. Coornaert, \emph{Mesures de Patterson-Sullivan sur le bord d'un espace hyperbolique au sens de Gromov},
Pacific J. Math. 159 (1993), 241--270. \MR{1214072}

\bibitem[CP96]{MR1425084}
M.~Coornaert and A.~Papadopoulos, \emph{R\'ecurrence de marches al\'eatoires et
  ergodicit\'e du flot g\'eod\'esique sur les graphes r\'eguliers}, Math.
  Scand. \textbf{79} (1996), no.~1, 130--152. \MR{1425084}

\bibitem[Ger77]{MR0461671}
P.~Gerl, \emph{Irrfahrten auf {$F_{2}$}}, Monatsh. Math. \textbf{84} (1977),
  no.~1, 29--35. \MR{0461671}
  


\bibitem[Gri80]{MR599539}
R.~I. Grigorchuk, \emph{Symmetrical random walks on discrete groups},
  Multicomponent random systems, Adv. Probab. Related Topics, vol.~6, Dekker,
  New York, 1980, pp.~285--325. \MR{599539}
  
  \bibitem[GdlH97]{MR1436550}
R.~Grigorchuk and P.~de~la Harpe, \emph{On problems related to growth, entropy,
  and spectrum in group theory}, J. Dynam. Control Systems \textbf{3} (1997),
  no.~1, 51--89. \MR{1436550}
  
  \bibitem[HH97]{MR1468105}
S.~Hersonsky and J.~Hubbard, \emph{Groups of automorphisms of trees and their
  limit sets}, Ergodic Theory Dynam. Systems \textbf{17} (1997), no.~4,
  869--884. \MR{1468105}
  
  
  \bibitem[Jae14]{MR3240930}
J.~Jaerisch, \emph{Fractal models for normal subgroups of {S}chottky groups},
  Trans. Amer. Math. Soc. \textbf{366} (2014), no.~10, 5453--5485. \MR{3240930}
  
  \bibitem[Kes59]{MR0109367}
H.~Kesten, \emph{Symmetric random walks on groups}, Trans. Amer. Math. Soc.
  \textbf{92} (1959), 336--354. \MR{0109367}
  
  \bibitem[Kes59a]{MR0112053}
H.~Kesten, \emph{Full {B}anach mean values on countable groups}, Math. Scand.
  \textbf{7} (1959), 146--156. \MR{0112053}
  
\bibitem[Led01]{MR1832436}  
F. Ledrappier, 
\emph{Some asymptotic properties of random walks on free groups}, Topics in probability and Lie groups: boundary theory, 
CRM Proc. Lecture Notes, vol.~28, Amer. Math. Soc., Providence, 2001, pp.~117--152. 
\MR{1832436} 

\bibitem[OW07]{MR2338235}
R.~Ortner and W.~Woess, \emph{Non-backtracking random walks and cogrowth of
  graphs}, Canad. J. Math. \textbf{59} (2007), no.~4, 828--844. \MR{2338235}

\bibitem[RT15]{MR3350111}
T.~Roblin and S.~Tapie, \emph{Critical exponent and bottom of the spectrum in
  pinched negative curvature}, Math. Res. Lett. \textbf{22} (2015), no.~3,
  929--944. \MR{3350111}
  
  \bibitem[Sho91]{MR1170365}
H.~Short, \emph{Quasiconvexity and a theorem of {H}owson's}, Group theory from
  a geometrical viewpoint ({T}rieste, 1990), World Sci. Publ., River Edge, NJ,
  1991, pp.~168--176. \MR{1170365}

\bibitem[Swe01]{MR1804703}
E.~L. Swenson, \emph{Quasi-convex groups of isometries of negatively curved
  spaces}, Topology Appl. \textbf{110} (2001), no.~1, 119--129, Geometric
  topology and geometric group theory (Milwaukee, WI, 1997). \MR{1804703}

\bibitem[Sul87]{MR882827}
D.~Sullivan, \emph{Related aspects of positivity in {R}iemannian geometry}, J.
  Differential Geom. \textbf{25} (1987), no.~3, 327--351. \MR{882827}
  
\bibitem[Woe00]{MR1743100}    
W. Woess, \emph{Random walks on infinite graphs and groups},
Cambridge Tracts in Mathematics, vol.~138, Cambridge University Press, Cambridge, 2000.
\MR{1743100}

\end{thebibliography}
\end{document}